\documentclass{article}

\newcommand{\inserttitle}{the prime grid contains arbitrarily large empty polygons}

\usepackage[T1]{fontenc}
\usepackage{setspace}
\newif\ifsitem
\newcommand{\setupstar}{%
  \global\sitemfalse
  \let\origmakelabel\makelabel
  \def\sitem{\global\sitemtrue\item}
  \def\makelabel##1{%
    \origmakelabel{\ifsitem\llap{\raisebox{0.17ex}{$\mathbf{\ast}$}\:}\fi##1}%
    \global\sitemfalse}%
} 

\usepackage[letterpaper, portrait, margin=1.2in, footnotesep=1.5\baselineskip]{geometry}
\usepackage{booktabs}
\usepackage[dvipsnames]{xcolor}
\usepackage{todonotes}
\usepackage{contour}
    \contourlength{0.5pt}
\usepackage{tocloft}

\usepackage{imakeidx}
    \makeindex
\usepackage{environ} 
\usepackage[framemethod=tikz]{mdframed}
\usepackage{wrapfig}
\usepackage{mathtools}
    \usetikzlibrary{calc, math, cd, arrows.meta, braids, decorations.pathreplacing, decorations.markings, decorations.pathmorphing, bending, shapes}
    \usepackage{tikz-3dplot}
    \tikzset{vertex/.style={draw, shape=circle, inner sep=1.5pt, minimum size=4pt}}
    \tikzset{<->/.tip={Latex}}
    \tikzset{smallnode/.style={every node/.style={draw, fill=black, shape=circle, inner sep=0pt, minimum size=4pt}}}
    \tikzset{drawnode/.style={fill,shape=circle,inner sep=0pt, minimum size=3pt}}
\usepackage{pgfornament}
\usepackage{fourier-orns} 
\usepackage{stmaryrd}
\usepackage{halloweenmath}
\usepackage{faktor}
\usepackage[shortlabels]{enumitem}
    \newlength{\circlabelwidth}
    \setlist{nosep}
    \setlist[enumerate]{label=\textup{\arabic*.}}
    \newlist{subprob}{enumerate}{2}
        \setlist[subprob,1]{label={(\roman*)}}
        \setlist[subprob,2]{label={(\arabic*)}}
    \setlist[itemize]{labelindent=10pt,labelwidth=\circlabelwidth,leftmargin=!,label=$\circ$}
    \newlist{problems}{enumerate}{3}
        \setlist[problems,1]{before=\setupstar,label=\textup{\arabic*.}, itemsep=2pt, topsep=8pt,ref=\textup{\arabic*}}
        \setlist[problems,2]{before=\setupstar,label=(\alph*),parsep=0pt}
        \setlist[problems,3]{before=\setupstar,label=(\roman*),parsep=0pt}

\usepackage{fancyhdr}
\usepackage{dopestyle}

\makeatletter
    \renewcommand\@makefntext[1]{\leftskip=0em\hskip-0em\@makefnmark\,#1}
\makeatother

\surroundwithmdframed[skipabove=0.5\baselineskip, skipbelow=0.5\baselineskip, leftmargin=3pt, rightmargin=3pt, innerleftmargin=7pt, innerrightmargin=7pt, roundcorner=10pt, linewidth=2pt, linecolor=red!40, backgroundcolor=red!5]{headsup}

\surroundwithmdframed[topline=false, bottomline=false, innertopmargin=0pt, innerbottommargin=0pt, innerleftmargin=2pt, innerrightmargin=2pt, linewidth=0.2mm]{quote}

\surroundwithmdframed[topline=false, bottomline=false, innertopmargin=2.5pt, innerbottommargin=2.5pt, innerleftmargin=-10pt, leftmargin=-10pt, innerrightmargin=-10pt, rightmargin=7.7pt, linewidth=0.4mm]{quote}

\NewEnviron{method}{
    \parbox{\textwidth}{
        \textbf{\textsc{Method}}
        \begin{mdframed}[innerleftmargin=4pt,innerrightmargin=4pt,skipabove=3pt,skipbelow=0pt]
            \BODY
        \end{mdframed}
    }
}

\theoremstyle{itcaps}
\newtheorem{theorem}{Theorem}
\newtheorem*{theorem*}{Theorem}

\newtheorem{proposition}[theorem]{Proposition}
\newtheorem{lemma}[theorem]{Lemma}

\theoremstyle{solved}

\theoremstyle{caps}

\newtheorem*{definition*}{Definition}
\newtheorem{exampleprimitive}[theorem]{Example}

    \crefname{exercise}{Exercise}{Exercises}
\newtheorem{problem}{Problem}

\theoremstyle{remark}

\numberwithin{equation}{section}

\newcommand*{\newword}[2][]{\emph{#2}\index{%
    \ifx&#1&%
       #2%
    \else%
       #1%
    \fi}%
} 
\newcommand*{\oldword}[2][]{#2\index{%
    \ifx&#1&%
       #2%
    \else%
       #1%
    \fi}%
} 

\renewcommand*{\P}{\mathbb{P}}

\DeclareMathOperator{\conv}{conv}

\let\phi\varphi

\let\epsilon\varepsilon

\let\oldchi\chi
\renewcommand{\chi}{\raisebox{1pt}{$\oldchi$}}

\title{\inserttitle}
\date{}

\allowdisplaybreaks

\usepackage{caption}
\begin{document}

\setlength{\abovedisplayskip}{6pt plus 2pt minus 4pt}
\setlength{\abovedisplayshortskip}{1pt plus 3pt}
\setlength{\belowdisplayskip}{6pt plus 2pt minus 4pt}
\setlength{\belowdisplayshortskip}{6pt plus 2pt minus 2pt}

\setlength{\parindent}{10pt}

\usetikzlibrary{patterns}

\theoremstyle{itcaps}
\newtheorem*{main-theorem}{Main theorem}

\definecolor{sea green}{RGB}{56,199,137}
\definecolor{sea blue}{RGB}{56,190,199}

\thispagestyle{plain}

\makeatletter

\null
\vspace{-1.5em}

\noindent
    \hfill%
    \vbox{%
        \hsize0.9\textwidth
        \linewidth\hsize
        {
          \hrule height 2\p@
          \vskip 0.25in
          \vskip -\parskip%
        }
        \centering
        {\LARGE\sc the prime grid contains arbitrarily\\
    large empty polygons\par}
        {
          \vskip 0.26in
          \vskip -\parskip
          \hrule height 2\p@
          \vskip 0.10in%
        }
    }%
    \hfill\null%
\makeatother

\vspace{0.5em}

\begin{center}\large
    \scshape Travis Dillon
    %
    \\[2em]
\end{center}

\begin{abstract}
    This paper proves a 2017 conjecture of De Loera, La Haye, Oliveros, and Rold\'an-Pensado that the ``prime grid'' $\big\{(p,q) \in \Z^2 : \text{$p$ and $q$ are prime}\big\} \subseteq \R^2$ contains empty polygons with arbitrarily many vertices. This implies that no Helly-type theorem is true for the prime grid.
\end{abstract}

\section{Introduction}

Finding structure and patterns in the prime numbers has been a pastime and pursuit of mathematicians for ages---a quest that is far from complete, but with many exciting recent advances. This paper investigates a specific geometric pattern, not in the primes themselves, but in their two-dimensional Cartesian product.

Using $\P$ to denote the set of prime numbers, we let $\P^2 = \P \times \P$ denote the set of points in $\Z^2$ in which both coordinates are primes. In a paper on discrete geometry, De Loera, La Haye, Oliveros, and Rold\'an-Pensado \cite{Helly-algebraic-subsets} conjectured that $\P^2$ contains arbitrarily large \emph{empty polgyons}: polygons which intersect $\P^2$ at, and only at, their vertices. (See \cref{fig:empty-polygons} for an illustration.)

How did a conjecture on primes come to appear in a geometry paper? One can only guess that the answer lies partly in whim: The primes attract fascination across mathematical boundaries. But the more serious reason is because of the relationship between empty polytopes and Helly-type theorems in geometry.

Helly's theorem is a fundamental property of convex sets in Euclidean space, first proved by Eduard Helly in 1913, though the first published proof is by Johann Radon \cite{Radon-helly-proof}, with Helly's proof following two years later \cite{Helly}.

\begin{theorem*}[Helly]
    If the intersection of a finite family of convex sets in $\R^d$ is empty, then there is a subfamily of at most $d+1$ sets whose intersection is empty.
\end{theorem*}

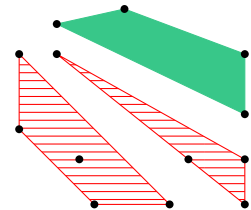
\begin{wrapfigure}{r}{0.25\textwidth}
    \centering
    \begin{tikzpicture}[scale=2]
        \foreach \x/\y [count=\z] in {1.5/0.6, 1.5/1, 0.7/1.3, 0.25/1.2,
            0.25/1, 1.5/0, 1.5/0.3,
            0/0.5, 0/1, 1/0, 0.5/0}
            \coordinate (\z) at (\x,\y) {};
        \fill[draw=sea green, fill=sea green, fill opacity=0.45] (1) -- (2) -- (3) -- (4) -- (1);
        \fill[draw=red, pattern=horizontal lines, pattern color=red, fill opacity=0.3] (5) -- (6) -- (7) -- (5);
        \fill[draw=red, pattern=horizontal lines, pattern color=red, fill opacity=0.3] (8) -- (9) -- (10) -- (11) -- (8);
        \foreach \x in {1,2,...,11}
            \node[drawnode] at (\x) {};
        \node[drawnode] at (1.125,0.3) {};
        \node[drawnode] at (0.4,0.3) {};
    \end{tikzpicture}
    \caption{An empty polygon (in solid green) and two non-empty polygons (in dashed red).}
    \label{fig:empty-polygons}
\end{wrapfigure}

This is a striking property of convex sets. To put it another way: In order to guarantee that the intersection of a finite family of convex sets in $\R^d$ is \emph{nonempty}, it suffices to check that the intersections of all subfamilies of size $d+1$ are nonempty.

Helly's theorem is the basis for a flourishing area of research, one founded on poking and prodding the theorem in all different directions: weakening and strengthening the hypothesis and conclusion, abstracting the notion of convexity, adding quantitative measurements of the size of the intersection, and more. The first survey of this area was in 1963 by Danzer, Gr\"unbaum, and Klee \cite{Danzer-survey}, and a flurry of results in the last 20 years has led to two more surveys in the last 10 years \cite{barany-helly-survey,Soberon-Helly-survey}.

For us, the most important direction will be discrete versions of Helly's theorem, the first of which was proven by Doignon in 1973 \cite{Doignon} (and re-proved a few years later by Bell \cite{bell-doignon} and Scarf \cite{scarf-doignon}):

\begin{theorem*}[Doignon]
    If, in a finite family of convex sets in $\R^d$, the intersection of any subfamily of $2^d$ sets contains a point in $\Z^d$, then the intersection of the entire family contains a point in $\Z^d$.
\end{theorem*}

So a Helly-type theorem holds for the integer lattice, though it requires checking significantly larger subfamilies. In general, for any given $S \subseteq \R^d$, we can ask whether a Helly-type theorem holds for $S$. The general framework for this question is that of \emph{$S$-Helly numbers}:

\begin{definition*}
    The \emph{Helly number} of a set $S \subseteq \R^d$, denoted $h(S)$, is the minimal number (if it exists), such that the following statement is true: In any finite family of convex sets in $\R^d$, if the intersection every subfamily of at most $h(S)$ sets contains a point of $S$, then the intersection of the entire family contains a point of $S$. (If no such number exists, we write $h(S) = \infty$.)
\end{definition*}

In other words, $h(S)$ is the minimal size of the subfamily condition in a Helly-type theorem for $S$. Helly's theorem says that $h(\R^d) \leq d+1$, while Doignon's theorem says that $h(\Z^d) \leq 2^d$.

As it turns out, Helly numbers for discrete sets are intimately related to the empty polytopes formed by those sets. (A set in $\R^d$ is called \emph{discrete} if it has no accumulation points.) This connection was first described by Hoffman in 1979 \cite{Hoffman-binding-constraints} (and later expanded on by Averkov and Weismantel \cite{Averkov-empty-polytopes}):

\begin{theorem*}[Empty polytopes and Helly-type theorems]
    If $S \subseteq \R^n$ is discrete, then $h(S)$ is equal to the maximal number of vertices of an empty polytope in $S$. \textup{(}If there is no maximum, then $h(S) = \infty$.\textup{)}
\end{theorem*}

With this result in hand, determining Helly numbers becomes a much more approachable task. For example, let's prove Doignon's theorem by showing that no collection of $2^d + 1$ points in $\Z^d$ form the vertices of an empty polytope. Among any set $X$ of $2^d + 1$ integer points, there are two points, say $x,y \in \Z^d$, which have the same parity in every coordinate---but then $\frac{1}{2}(x+y)$ is an integer point contained in $\conv(X)$ but is not a vertex, so $\conv(X)$ is not empty. Therefore, $h(\Z^d) \leq 2^d$. (On the other hand, the unit hypercube $[0,1]^d$ \emph{is} empty in $\Z^d$ and has $2^d$ vertices, so $h(\Z^d) \geq 2^d$.)

In this notation, the conjecture of De Loera, La Haye, Oliveros, and Rold\'an-Pensado is that $h(\P^2) = \infty$; in other words, that there is no Helly-type theorem for $\P^2$. If their conjecture were false, that would mean that there is a $k \in \Z$ such that: for any finite family of convex sets in $\R^d$, if every subfamily of size $k$ contains a point in its intersection in which both coordinates are prime, then the intersection of the entire family contains such a point. A statement like this seems ridiculous, which makes it eminently believable that $h(\P^2) = \infty$, as conjectured in \cite{Helly-algebraic-subsets}.

On the other hand, Helly numbers aren't always easily predictable. For example, Ambrus, Balko, Frankl, Jung, and Nasz\'odi \cite{Ambrus-exponential-lattice} recently proved that the Helly number of the exponential grid is finite: $h\big(\{2^n : n \in \N\}^2\big) = 5$. And yet this, phrased as a Helly-type theorem, seems just as strange as a Helly-type theorem for the prime grid. However, the authors of \cite{Ambrus-exponential-lattice} also prove that Fibonacci grid has infinite Helly number: $h\big(\{ F_n : n \in \N \}^2 \big) = \infty$ (where $F_n$ is the $n$th Fibonacci number). The country of Helly-type theorems can be a strange land.

Up to now, the best known bound the Helly number of the prime grid is that $h(\P^2) \geq 14$, obtained by a computer search \cite{Summers-prime-helly}. In this paper, we confirm the conjecture of De Loera, La Haye, Oliveros, and Rold\'an-Pensado and prove that there is no Helly-type theorem for the prime grid:

\begin{main-theorem}
    The prime grid $\P^2$ contains empty polygons with arbitrarily many vertices.
\end{main-theorem}

\cref{sec:main-thm} comprises a proof of this theorem, and \cref{sec:problems} concludes with some related open problems.

\section{Proof of the main theorem}\label{sec:main-thm}

To prove something about the prime grid, we'll need to use some number theory. The question is, what kind? The authors of \cite{Helly-algebraic-subsets} suggest that determining $h(\P^2)$ may be related to the Gilbreath--Proth conjecture on the iterated differences between primes. Although there has been some recent work on it \cite{random-gilbreath}, we seem far from a resolution of this conjecture.

Fortunately, we won't need to resolve the Gilbreath--Proth conjecture in order to determine $h(\P^2)$. Instead, we will rely on the spectacular Maynard--Tao theorem on gaps between primes \cite{Maynard-Tao}.

Our overall strategy is to look for empty polygons close to the main diagonal (the line $y=x$) of the prime grid. In doing so, we can reduce the problem to verifying a condition on the ratios between consecutive prime gaps. The lemma that accomplishes this appeared previously in \cite{Dillon-helly-boxes}; we reproduce it here in a slightly different form.

\begin{lemma}\label{thm:product-sets-lower-bound}
Given a discrete set $A \subseteq \R$, list its elements in increasing order as $a_1 < a_2 < \cdots$. If there are $k$ consecutive elements of $A$ such that
\[ 
    \frac{a_{i+3} - a_{i+2}}{a_{i+2} - a_{i+1}}
    <
    \frac{a_{i+2} - a_{i+1}}{a_{i+1} - a_i}
\]
for each $i \in \{t, t+1, \dots, t+k-3\}$, then $h(A^2) \geq k+1$.
\end{lemma}
\begin{proof}
We define the set
\[
     X = \{(a_b,a_b),(a_{b+1},a_{b+1})\} \cup \{(a_i,a_{i+1}) : b \leq i \leq b+k-1\}
\]
and the polygon $P = \conv(X)$. (See \cref{fig:near-diag-polygon} for an illustration of $P$.) By construction, $P$ is empty, and the condition on the ratios of differences guarantees that each point in $X$ is a vertex of $P$. Therefore $h(A^2) \geq |X| = k+1$.
\end{proof}

\begin{figure}[b]
    \centering
    \begin{tikzpicture}[scale=0.75]
        \draw[blue, opacity=0.6, thick, densely dashed] (-0.5,-0.5) -- (5.5,5.5);
        \draw[sea green,very thick,fill=sea green, fill opacity=0.45] (0,0) -- (0,0.75) -- (0.75,2.25) -- (2.25,4) -- (4,5) -- (0.75,0.75) -- (0,0);
        \foreach \x [count=\count] in {0,0.75,2.25,4,5}
            {
            \node at (\x,-1) {$a_\count$};
            \node at (-1,\x) {$a_\count$};
            \foreach \y in {0,0.75,2.25,4,5} 
                \node[draw, shape=circle, fill=black, inner sep=0pt, minimum size=2.75pt] at (\x,\y) {};
            }
    \end{tikzpicture}
    \caption{The polygon $P$ in the proof of \cref{thm:product-sets-lower-bound}, with $k=5$.}
    \label{fig:near-diag-polygon}
\end{figure}
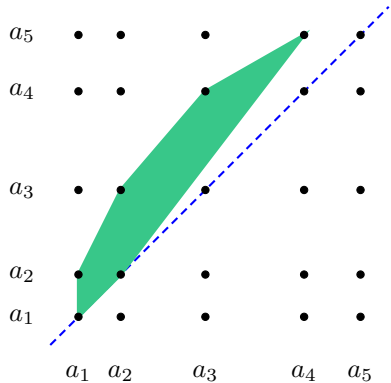

To prove the main theorem, we now need to verify that the ratio condition of \cref{thm:product-sets-lower-bound} holds for arbitrarily long sequences of consecutive primes. We will use the following result, which was developed from the Maynard--Tao theorem \cite{Maynard-Tao} in order to prove a conjecture of Erd\H{o}s and Tur\'an \cite{erdos-prime-gap}. In the following theorem, a set of integers $\{a_1,\dots,a_m\}$ is called \emph{admissible} if for every prime $p$, there is an $n \in \Z/p\Z$ such that $n \notin \{ a_i \text{\,mod\,} p : 1 \leq i \leq m\}$.

\begin{theorem}[Banks--Frieberg--Turnage-Butterbaugh \cite{consecutive-primes}]\label{thm:consecutive-primes}
    For every $k \in \N$ there is an $m \in \N$ such that, for all admissible sets $a_1,a_2,\dots, a_m \in \N$, there is a subset $\{b_1,b_2,\dots,b_k\} \subseteq \{a_1,a_2,\dots,a_m\}$ and infinitely many integers $n$ such that $n+b_1, n+ b_2, \dots, n+b_k$ are consecutive primes.
\end{theorem}

\begin{lemma}\label{thm:prime-gap-ratios}
    For every $k$, there is a set of $k$ consecutive primes $p_1,p_2,\dots,p_k$ such that
    \[
        \frac{p_{i+3}-p_{i+2}}{p_{i+2}-p_{i_1}}
        < \frac{p_{i+2}-p_{i+1}}{p_{i+1}-p_i}
    \]
    for every $i \in \{1,2,\dots,k-3\}$. There is also a set of consecutive primes $q_1, q_2, \dots, q_k$ such that the opposite inequality
    \[
        \frac{q_{i+3}-q_{i+2}}{q_{i+2}-q_{i_1}}
        > \frac{q_{i+2}-q_{i+1}}{q_{i+1}-q_i}
    \]
    holds for every $i \in \{1,2,\dots,k-3\}$.
\end{lemma}
\begin{proof}
    We start by proving the second inequality. The idea is to apply \cref{thm:consecutive-primes} and choose a sequence $(a_i)$ that grows so quickly that the inequality of ratio differences holds for any subset $\{b_1,\dots,b_k\} \subseteq \{a_1,\dots,a_m\}$.
    
    Given $k$, let $m$ be as in \cref{thm:consecutive-primes}, and set $a_i = 3^{2^i}$ for $1 \leq i \leq m$. This is an admissible set, because $3^{2^i}\not\equiv 0 \pmod{p}$ if $p\neq 3$ and $3^{2^i} \not\equiv 1 \pmod{3}$ for every $i \in \N$. \cref{thm:consecutive-primes} guarantees that there are indices $r(1) < r(2) < \dots < r(k)$ and an integer $n$ such that $n+a_{r(1)},\, n + a_{r(2)},\, \dots,\, n + a_{r(k)}$ are $k$ consecutive primes. We will call these numbers $q_1, q_2, \dots, q_k$, respectively. Using the fact that $\frac{a-1}{b-1} > \frac{a}{b}$ whenever $a > b$, we have
    \[
        \frac{q_{i+3} - q_{i+2}}{q_{i+2}- q_{i+1}}
        = \frac{3^{2^{r(i+2)}} \big(3^{2^{r(i+3)}-2^{r(i+2)}} - 1\big)}{3^{2^{r(i+1)}} \big(3^{2^{r(i+2)}-2^{r(i+1)}} - 1\big)}
        > \frac{3^{2^{r(i+2)}} \cdot 3^{2^{r(i+3)}-2^{r(i+2)}}}{3^{2^{r(i+1)}} \cdot 3^{2^{r(i+2)}-2^{r(i+1)}}}
        = 3^{2^{r(i+3)} - 2^{r(i+2)}}.
    \]
    On the other hand, a lower bound on the denominator yields that
    \[
        \frac{q_{i+2} - q_{i+1}}{q_{i+1}- q_{i}}
        < \frac{3^{2^{r(i+2)}} - 3^{2^{r(i+1)}}}{3^{2^{r(i+1)}-1}}
        < 3^{2^{r(i+2)} - 2^{r(i+1)} + 1}. 
    \]

    Using the fact that $2^{i+1} - 2^i > 2^i - 2^j + 1$ for any $j\geq 1$, we can chain everything together:
    \[
        \frac{q_{i+2} - q_{i+1}}{q_{i+1}- q_{i}}
        < 3^{2^{r(i+2)} - 2^{r(i+1)} + 1}
        < 3^{2^{r(i+2)} - 2^{1} + 1}
        < 3^{2^{r(i+3)} - 2^{r(i+2)}}
        < \frac{q_{i+3} - q_{i+2}}{q_{i+2}- q_{i+1}}.
    \]
    This proves the second claim in \cref{thm:prime-gap-ratios}.

    To prove the first claim, use numbers $a_i = -3^{2^i}$. This reverses the order of the gaps between primes, which also reverses the inequality between the consecutive gap ratios.
\end{proof}

Together, \cref{thm:product-sets-lower-bound} and \cref{thm:prime-gap-ratios} prove the Main Theorem.

\section{Related problems}\label{sec:problems}

In this area, there are open problems aplenty, because coming up with a new problem is as easy as choosing a point set. Here are a few specific problems.

In the same paper where they conjecture that $h(\P^2) = \infty$, De Loera, La Haye, Oliveros, and Rold\'an-Pensado observe that $h\big((\Z\setminus \P)^2\big)$ is finite. We can provide a short proof of this fact using the ``flatness theorem\slimcomma'' first proved by Khinchine in 1948 \cite{khinchine}. Arun and Dillon used a similar approach to bound the number of vertices of certain classes of so-called \emph{hollow polytopes} \cite{Arun-empty-polytopes}.

A \emph{lattice hyperplane} of a lattice $L$ is a hyperplane that passes through a point in $L$.

\begin{theorem*}[Flatness theorem]
    For any convex body $K$ in $\R^d$ that does not contain any lattice points, there is a vector $v$ such that at most $f(d)$ lattice hyperplanes orthogonal to $v$ intersect $K$.
\end{theorem*}

\newpage

\begin{proposition}\label{thm:prime-complement-helly}
    Both $h\big((\Z \setminus \P)^2\big)$ and $h(\Z^2 \setminus \P^2)$ are finite.
\end{proposition}
\begin{proof}
    If $P$ is an empty polytope in $(\Z \setminus \P)^2$, then only the vertices of $P$ can intersect $(4\Z)^2$. Thus, $\frac{1}{4} P$ is a convex body whose interior does not intersect $\Z^2$. By the flatness theorem, there is a direction $v$ with at most $4\, f(2)$ lattice lines of $\Z^2$ that are orthogonal to $v$ and intersect $P$. Each line contains at most 2 vertices of $P$, and each vertex is contained in one line, so $P$ has at most $8\,f(2)$ vertices.

    The same argument holds for $\Z^2 \setminus \P^2$.
\end{proof}

Corollary 1 in \cite{flatness-dim-2} implies that $f(2) = 3$, so $h(S) \leq 24$ for $S = (\Z \setminus \P)^2$ and $S = \Z^2 \setminus \P^2$. Using the results in \cite{flatness-dim-2} and the slightly more sophisticated notion of ``lattice width'' in the proof of \cref{thm:prime-complement-helly} improves this bound to $h(S) \leq 18$ for both sets.

\begin{problem}
    Evaluate $h\big((\Z \setminus \P)^2\big)$ and $h(\Z^2\setminus \P^2)$. What are $h\big((\Z \setminus \P)^d\big)$ and $h(\Z^d \setminus \P^d)$ for $d\geq 3$?
\end{problem}

As for other point sets, Ambrus, Balko, Frankl, Jung, and Nasz\'odi \cite{Ambrus-exponential-lattice} recently proved that the exponential grid $\{\alpha^n : n \in \N\}^2$ has finite Helly number for any $\alpha > 1$. Arun and Dillon \cite{Arun-empty-polytopes} slightly improved the upper bound, but both papers rely on an argument that doesn't easily generalize to higher dimensions.

\begin{problem}
    For which $\alpha > 1$ and $d \geq 3$ is $h\big( \{\alpha^n : n \in \N\}^d\big)$ finite? In particular, does $\{2^n : n \in \N\}^3$ have finite Helly number?
\end{problem}

\null
\begin{center}
    \large\scshape acknowledgments
\end{center}
\noindent
I gladly thank Pablo Sober\'on for first discussing this problem with me, Ashwin Sah for an insightful conversation about it, and Henry Cohn for reviewing a draft of this paper. This work was partially supported by a National Science Foundation Graduate Research Fellowship under Grant No. 2141064.\vspace{1.5\baselineskip}

\let\OLDthebibliography\thebibliography
\renewcommand\thebibliography[1]{
  \OLDthebibliography{#1}
  \setlength{\parskip}{0pt}
  \setlength{\itemsep}{3pt plus 0.3ex}
}

\bibliography{bibliography.bib}
\bibliographystyle{amsplain-nodash}

\vspace{2\baselineskip}

\noindent
{\small \textsc{Department of Mathematics, Massachusetts Institute of Technology, Cambridge, MA, USA}}\\
\textit{email:} \texttt{dillont@mit.edu}

\end{document}